\documentclass{amsart}

\newcommand{\beqa}{\begin{eqnarray*}}
\newcommand{\eeqa}{\end{eqnarray*}}
\newcommand{\beqn}{\begin{eqnarray}}
\newcommand{\eeqn}{\end{eqnarray}}

\newcommand{\ra}{\rightarrow}

\newcommand{\bT}{\mathbb T}

\newcommand{\ov}{\overline}

\newcommand{\la}{\lambda}

\newcounter{cnt1}
\newcounter{cnt2}
\newcounter{cnt3}
\newcommand{\blr}{\begin{list}{$($\roman{cnt1}$)$}
 {\usecounter{cnt1} \setlength{\topsep}{0pt}
 \setlength{\itemsep}{0pt}}}
\newcommand{\bla}{\begin{list}{$($\alph{cnt2}$)$}
 {\usecounter{cnt2} \setlength{\topsep}{0pt}
 \setlength{\itemsep}{0pt}}}
\newcommand{\bln}{\begin{list}{$($\arabic{cnt3}$)$}
 {\usecounter{cnt3} \setlength{\topsep}{0pt}
 \setlength{\itemsep}{0pt}}}
\newcommand{\el}{\end{list}}

\newtheorem{thm}{Theorem}[section]

\newtheorem{cor}[thm]{Corollary}

\newtheorem{Def}[thm]{Definition}

\newtheorem{rem}[thm]{Remark}
\newcommand{\Rem}{\begin{rem} \rm}
\newcommand{\bdfn}{\begin{Def} \rm}
\newcommand{\edfn}{\end{Def}}

\newcommand{\ba}{\begin{array}}
\newcommand{\ea}{\end{array}}

\newtheorem*{ack}{Acknowledgements}

\begin{document}
\sloppy

\title{Algebraic reflexivity of the set of $n$-isometries on $C(X,E)$}
\author{ABDULLAH BIN ABU BAKER}

\address[Abdullah Bin Abubaker]{Department of Mathematics and Statistics\\
Indian Institute of Technology, Kanpur \\
India. \textit{E-mail~:} \textit{abdullah@iitk.ac.in}}

.


\begin{abstract}

We prove that if the group of isometries of $C(X,E)$ is
algebraically reflexive, then the group of $n$-isometries is also
algebraically reflexive. Here, $X$ is a compact Hausdorff space
and $E$ is a Banach space. As a corollary to this, we establish
the algebraic reflexivity of the set of generalized bi-circular
projections on $C(X,E)$. This answers a question raised in
\cite{DR}.

\end{abstract}

\maketitle

\section{Introduction}

Let $X$ be a compact Hausdorff space and $E$ be a Banach space. We
denote by $\mathcal{G}(E)$ the group of all surjective isometries
of $E$. Suppose that $E$ has trivial centralizer. Then any isometry $T$
of $C(X,E)$ is of the form $Tf(x)= \tau(x)(f(\phi(x)))$ for $x \in X$ and
$f \in C(X,E)$, where $\tau : X \ra \mathcal{G}(E)$ continuous in strong
operator topology and $\phi$ is a homeomorphism  of $X$ onto itself.
Let $\mathcal{G}^n(C(X,E))= \{ T \in
\mathcal{G}(C((X,E))): T^n=I\}$. Any $T \in \mathcal{G}^n(C(X,E))$ is
called an $n$-isometry. It can be easily proved that
$T \in \mathcal{G}^n(C(X,E))$ if and
only if there exist a homeomorphism $\phi$ of $X$ such that
$\phi^n(x)=x$ for all $x \in X$, a map $\tau : X \ra
\mathcal{G}(E)$ satisfying $\tau(x) \circ
\tau(\phi(x)) \circ...\circ \tau(\phi^{n-1}(x))=I$, where $I$
denotes the identity map and $T$ is given by $Tf(x)=
\tau(x)(f(\phi(x)))$, for all $x \in X$ and $f \in C(X,E)$.

Let $B(E)$ be the set of all bounded linear operators on $E$ and
$S$ is any subset if $B(E)$. The Algebraic closure
$\ov{\mathcal{S}}^a$ of $S$ is defined as follow: $T \in
\ov{\mathcal{S}}^a$ if for every $e \in E$ there exists $T_e \in
S$ such that $T(e) = T_e(e)$. $S$ is said to algebraically
reflexive if $S = \ov{\mathcal{S}}^a$. Let $\bT$ denotes the unit
circle. A linear projection $P: E \ra E$ is said to be a
generalized bi-circular projection (GBP for short) if for some
$\la \in \bT$, $ P + \la (I-P)$ is an isometry.

It was proved by Dutta  and Rao, see \cite{DR}, that for a compact Hausdorff space
$X$, if $\mathcal{G}(C(X))$ is algebraically reflexive then
$\mathcal G^2(C(X))$ is also algebraically reflexive. In this note we
prove this result for vector valued continuous functions and for
any $n \geq 2$. We also answer a question raised by them about the
algebraic reflexivity of the set of GBP's of $C(X,E)$.

\section{Main Results}

\begin{thm}\label{main}
Let $X$ be a compact Hausdorff space and $E$ be a Banach space. If
$\mathcal{G}(C(X,E))$ is algebraically reflexive, then
$\mathcal{G}^n(C(X,E))$ is also algebraically reflexive.
\end{thm}

\begin{proof}
Let $T \in \ov{\mathcal{G}^n(C(X,E))}^a$. Then, for each $f \in
C(X,E)$ we have $Tf(x)= \tau_f(x)(f(\phi_f(x)))$ where $\tau_{f} :
X \ra \mathcal{G}(E)$ is continuous in s.o.t satisfying $\tau_f(x)
\circ \tau_f(\phi(x)) \circ...\circ \tau_f(\phi^{n-1}(x))=I$ and
$\phi_f$ is a homeomorphism of $X$ such that $\phi^n_f(x)=x$ for
all $x \in X$. By the algebraic reflexivity of
$\mathcal{G}(C(X,E))$, $T$ is an isometry and hence $Tf(x)= \tau
(x)(f(\phi(x)))$ where $\tau(x) \in \mathcal{G}(E)$ for each $x$
and $\phi$ is a homeomorphism of $X$. To show that $T^n=I$ we need
to show that $\phi^{n}(x)=x$ and $\tau(x) \circ \tau(\phi(x))
\circ...\circ \tau(\phi^{n-1}(x))=I$.

Firstly suppose that $f$ is a rank one tensor product, i.e., $f=h
\bigotimes e$, where $h$ is a strictly positive function $C(X)$
and $e \in E$. Since $\tau (x)(f(\phi(x)))=
\tau_f(x)(f(\phi_f(x)))$, we get $\tau (x)(h(\phi(x))e)=
\tau_f(x)(h(\phi_f(x))e)$. As $\tau(x)$ and $\tau_f(x)$ are both
isometries we have $h(\phi(x))=h(\phi_f(x))$. Therefore,
$\tau(x)=\tau_f(x)$ for all $x \in X$.

Now consider any point $x \in X$.\\
\textbf{Case(i)}: $x=\phi(x)$. \\
Then $\phi^{n}(x)=\phi(\phi(...(\phi(x))...))(n~times)=x$. Choose
$h \in C(X)$ such that $0< h \leq 1 $ and $h^{-1}(1)=\{x\}$. For
$f=h \bigotimes e$, $e \in E$, evaluating $Tf$ at $x$ we get $
Tf(x)=\tau(x)(f(\phi(x)))=\tau(x)(h(x)e)=\tau(x)(e)=\tau_f(x)(f(\phi_f(x)))=
\tau_f(x)(h(\phi_f(x))e)$. As $\tau(x)$ and $\tau_f(x)$ are
isometries, by the choice of $h$ we have $\phi_f(x)=x$. Now, $I=
\tau_f(x) \circ \tau_f(\phi_f(x)) \circ...\circ
\tau_f(\phi_f^{n-1}(x))=\tau_f(x) \circ \tau_f(x) \circ...\circ
\tau_f(x)=\tau(x) \circ \tau(x) \circ...\circ \tau(x)$. \\
\textbf{Case(ii)}:  $\phi(x)\neq x$, $\phi^m(x)= x$, $m$ divides
$n$ and $\phi^s(x)\neq x$ $\forall~ s$ such that $ 1 \leq s <
m$.\\
As $m$ divides $n$, $n=mq$ for some positive integer $q$.
Therefore,
$\phi^n(x)=\phi^{mq}(x)=\phi^m(\phi^m(...(\phi^m(x)))...)(q~
times)=x$. Now, choose $h \in C(X)$ such that $1\leq h \leq m $
and $h^{-1}(1)=\{x\}$, $h^{-1}(2)=\{\phi(x)\}$,...,
$h^{-1}(m)=\{\phi^{m-1}(x)\}$. Let $f=h \bigotimes e$, for any $e
\in E$. Evaluating $Tf$ at $x$ we get $Tf(x)=
\tau(x)(f(\phi(x)))=\tau(x)(h(\phi(x))e)=\tau(x)(2e)=
\tau_f(x)(f(\phi_f(x)))=\tau_f(x)(h(\phi_f(x))e)$. This implies
that $\tau(x)(2e)=\tau_f(x)(h(\phi_f(x))e)$. As $\tau(x)$ and
$\tau_f(x)$ are isometries we get $h(\phi_f(x))=2$ and by the
choice of $h$ we get $\phi(x)=\phi_f(x)$. Similarly, applying $Tf$
at $\phi(x),...~,\phi^{n-2}(x)$ we get $\phi^p(x)= \phi_f^p(x)$
for $1\leq p\leq n-1$. Using the above and the fact that
$\tau(x)=\tau_f(x)$ for all $x \in X$ we have $\tau(x) \circ
\tau(\phi(x)) \circ...\circ \tau(\phi^{n-1}(x))=$ $\tau_f(x) \circ
\tau_f(\phi_f(x)) \circ...\circ \tau_f(\phi_f^{n-1}(x))=I$.\\
\textbf{Case(iii)}: $\phi(x)\neq x$, $\phi^{m}(x)= x$, $m$
does not divide $n$ and $\phi^{s}(x)\neq x$ $\forall~s$ such that $ 1 \leq s < m$.\\
Therefore, $\exists$ integers $r$ and $q$ such that
$n=mq+r,~0<r<m$. Now, choosing $h \in C(X)$ as in case(ii) and
proceeding in the same way we get $\phi^p(x)= \phi_f^p(x)$ for
$1\leq p\leq n-1$. Now,
$Tf(\phi^{n-1}(x))=\tau(\phi^{n-1}(x))(f(\phi^n(x)))=\tau(\phi^{n-1}(x))(h(\phi^n(x))e)
=\tau_f(\phi^{n-1}(x))(f(\phi_f(\phi^{n-1}(x))))
=\tau_f(\phi^{n-1}(x))(f(\phi_f(\phi^{n-1}_f(x))))=\tau_f(\phi^{n-1}(x))(f(\phi^n_f(x)))=
\tau_f(\phi^{n-1}(x))(f(x))=\tau_f(\phi^{n-1}(x))(h(x)e)=\tau_f(\phi^{n-1}(x))(e)$.
Thus,\\
$\tau(\phi^{n-1}(x))(h(\phi^n(x))e)=\tau_f(\phi^{n-1}(x))(e)$
which implies that $\phi^n(x)=x$.  But $\phi^m(x)= x$ implies that
$\phi^{mq}(x)=x$ and therefore
$x=\phi^n(x)=\phi^{r+mq}(x)=\phi^r(\phi^{mq}(x))=\phi^r(x)$,  a contradiction  because $r<m$.\\
\textbf{Case(iv)}: $\phi(x)\neq x$, $\phi^2(x)\neq x$, ...,
$\phi^{n-1}(x)\neq x$.\\
Choose $h$ such that $1\leq h \leq n $ and $h^{-1}(1)=\{x\}$,
$h^{-1}(2)=\{\phi(x)\}$,..., $h^{-1}(n)=\{\phi^{n-1}(x)\}$.
Proceeding the same way as in case (iii) we get $\phi^n(x)= x$ and
$\tau(x) \circ \tau(\phi(x)) \circ...\circ \tau(\phi^{n-1}(x))=I$.

\end{proof}

Our next corollary follows from a result by Jarosz and Rao, see
Theorem 7 in \cite{JR}, and Theorem~\ref{main}.

\begin{cor} \label{first}
Let $X$ be a first countable compact Hausdorff space and $E$ be a
uniformly convex Banach space such that $\mathcal{G}(E)$ is
algebraically reflexive. Then $\mathcal{G}^n(C(X,E))$ is also
algebraically reflexive.
\end{cor}

\begin{cor}
Let $X$ and $E$ be as in Corollary~\ref{first}. Then the set of
GBP's on  $C(X,E)$ is algebraically reflexive.
\end{cor}

\begin{proof}
Let the set of all GBP's on $C(X,E)$ be denoted by ${\mathcal P}$.
Let $P \in \ov{{\mathcal P}}^a$. Then for each $f \in C(X,E)$,
there exists $P_f \in {\mathcal P}$ such that $Pf = P_f f$. Thus
by \cite{BJ1}, for each $f$ there exists a
homeomorphism $\phi_f$ of $X$, $\phi^2_f (x) = x$ for all $x \in
X$ and $\tau_f : X \ra \mathcal G(E)$ satisfying $\tau_f(x) \circ
\tau_f(\phi_f(x)) = I$ such that $Pf(x) = \frac{1}{2} (f(x) +
\tau_f(x)( f(\phi_f(x))))$.

Thus for each $f,~ (2P - I)f(x) =  \tau_f(x)( f(\phi_f(x)))$ which
implies that $2P-I \in \ov{\mathcal G^2(C(X,E))}^a$. The
conclusion follows from Corollary~\ref{first}.

\end{proof}

\begin{ack}

The author would like to thank his supervisor Dr.S Dutta for his guidance
and many insightful discussions throughout the preparation of this article.

\end{ack}


\begin{thebibliography}{99}

\bibitem{DR} S.\ Dutta and T.\ S.\ S.\ R.\ K.\ Rao, {\em Algebraic
reflexivity of some subsets of the isometry group}, Linear Algebra
Appl. 429 (2008), no. 7, 1522|1527. MR2444339 (2009j:47153).

\bibitem{BJ1} F.\ Botelho and J.\ E.\ Jamison, {\em Generalized
bi-circular projections on $C(\Omega, X)$}, Rocky Mountain J.\
Math.\ 40 (2010), no. 1, 77|83. MR2607109 (2011c:47039).

\bibitem{JR} K.\ Jarosz and T.\ S.\ S.\ R.\ K.\ Rao, {\em Local
surjective isometries of function spaces}, Math. Z. 243 (2003),
449|469. MR1970012 (2003m:46036).


\end{thebibliography}
\end{document}